\documentclass{article}
\usepackage{amsfonts,amsmath,amssymb,amsthm, enumerate}
\usepackage[hidelinks]{hyperref}
\usepackage{tikz}
\usepackage{graphicx} % Required for inserting images
\usepackage{geometry}
\newcommand{\PP}{\mathbb{P}}
\newcommand{\QQ}{\mathbb{Q}}
\newcommand{\CC}{\mathbb{C}}
\newcommand{\RR}{\mathbb{R}}
\newcommand{\ZZ}{\mathbb{Z}}
\newcommand{\op}[1]{\operatorname{#1}}

\newcommand{\val}{\operatorname{val}}

\newcommand{\MA}{\operatorname{MA}}
\newcommand{\Sk}{\operatorname{Sk}}
\newcommand{\St}{\operatorname{St}}

\newtheorem{theorem}{Theorem}

\newtheorem{prop}{Proposition}
\newtheorem{lemma}{Lemma}
\newtheorem{rem}{Remark}
\newtheorem{corollary}{Corollary}
\newtheorem{definition}{Definition}
\newtheorem{remark}{Remark}
\newtheorem{conjecture}{Conjecture}

\title{Valuatively independent bases for the Fermat family of cubic curves}

\author{Jakob Hultgren and Sohaib Khalid}
\begin{document}

\maketitle

\begin{abstract}
    Let $\pi:(X,L)\rightarrow \mathbb D^*$ be the Fermat family of cubic curves in $\mathbb P^2$. For each $k\geq 1$, we construct an explicit valuatively independent basis for $H^0(X,L^k)$ in terms of restrictions of sections in
    $H^0(\mathbb P^2,\mathcal O_\mathbb P^2(l))$. As a consequence, we get an explicit formula for the canonical tropical theta functions and the canonical cost function defined by \emph{any} valuatively independent bases of $H^0(X,L^k)$. We show that the canonical tropical theta functions can be described as fundamental solutions to a real Monge-Ampère operator and the canonical cost function can be described intrinsically in terms of the monodromy of a Hessian structure on the essential skeleton. Notably, the canonical tropical theta functions differ from the ones induced by a monomial basis and the canonical cost function differs from the one induced by the ambient projective space. 
\end{abstract}

\section{Introduction}
Given a polarised degenerating family $\pi:(X,L) \to \mathbb D^*$ of Calabi-Yau manifolds which is meromorphic over the punctured unit disc $\mathbb D^*$, a central problem in complex geometry is to determine the limiting behaviour of the canonical Ricci-flat Calabi-Yau metrics $\omega_t\in c_1(L_t)$ on the fibres $X_t$ as $t\to 0$. A very important special case is when the degeneration is a \emph{large complex structure limit point}, which is the subject of the \emph{Strominger--Yau--Zaslow (SYZ)} picture of mirror symmetry, which predict that for small $t$, the fiber $X_t$ admits a large subset $U_t$ carrying a fibration of special Lagrangian tori of $(X_t,\omega_t)$. 

A program initiated by Yang Li uses local and global stability estimates for real and complex Monge-Ampère equations together with techniques from non-Archimedean geometry to prove this and the main technical challenge in this program has been to establish existence and regularity of solutions to a real Monge-Ampère equation on the essential skeleton $\operatorname{Sk}(X,\pi)$ of the family (see \cite{LiNA, PS, BMPS, GY} and the references therein.)

Due to the close connections between the theory of real Monge-Ampère equations and the theory of optimal transport, it is  natural to attempt to formulate an optimal transport problem whose optimiser can be shown to solve this real Monge-Ampère equation. Indeed, this is the approach taken in \cite{HJMM,LiToric,AH}. More generally, the picture that emerged from this is that such an optimal transport problem should be formulated on $\operatorname{Sk}(X,\pi)\times B$ where $B$ is to be thought of as the essential skeleton  $\operatorname{Sk}(X^\vee,\pi^\vee)$ of a degenerating family $\pi^\vee:X^\vee\to \mathbb D^*$ mirror to $\pi:X\to\mathbb D^*$. In \cite{LiCanonicalBasis}, Yang Li proposed a framework inspired by ideas from the Gross-Siebert program, where canonical bases of $H^0(X,L^k), k\geq 1$ are identified with lattice points in the top-dimensional faces of $B$ (see, for example \cite{GHS}), whereby a \emph{valuative} cost function for the optimal transport problem appears as the limit of straightforward valuations in these basis elements 
\begin{equation}\label{eq:ValuativeCost} c_k(x,s) = k^{-1}\val_x(s), x\in \operatorname{Sk}(X,\pi), s\in H^0(X,L^k)\subset B. 
\end{equation}
As explained in \cite{LiCanonicalBasis}, a key property of the canonical basis needed to formalise the construction in \eqref{eq:ValuativeCost} is \emph{valuative independence}. However, showing the existence of a valuatively independent bases was up until recently an open problem. After the first version of the present note was made public, Blum-Liu \cite{BL} showed, using a sophisticated but non-constructive argument, that $H^0(X,L^k)$ admits a valuatively independent basis for sufficiently large and divisible $k$. At the same time, in a landmark achievement, Yang Li \cite{YangLiMetricSYZ} used the existence of such a basis to produce the desired special Lagrangian torus fibration $U_t\rightarrow \op{Sk}(X,\pi)$ on subset $U_t$ of asymptotically full volume, in other words such that $|U_t|/|X_t|\rightarrow 1$ where $|\cdot|$ denotes the Calabi-Yau volume measure. This proved what is often referred to as the \emph{metric} SYZ-conjecture. 

However, very little is known about the subsets $U_t$. Their shape is determined by regularity properties of the minimisers of the optimal transport problem. As the valuatively independent basis and thus the resulting cost function is not explicit, only limited information is available. In particular, while it is known that the sets $U_t$ can be picked so that their union (over $t$) of their images in $\op{Sk}(X,\pi)$ is a dense subset, it is not known whether this subset is connected, its complement has codimension larger than 1 or if it admits an explicit description. In order to prove a more precise, or effective, version of the SYZ conjecture (see for example Corollary~2 in \cite{AH}) more information about the valuative cost function is needed. 

The normalised valuations of basis elements $\theta_s=c_k(\cdot,s) = k^{-1}\op{val}(s)$ are often called \emph{tropical theta functions}. It is important to note that both the sets of tropical theta functions and the valuative cost function determined by \eqref{eq:ValuativeCost} are independent of the choice of valuatively independent basis (see Lemma~6.3 in \cite{BL}, Lemma~3.4 in \cite{YangLiMetricSYZ} and the end of Section~\ref{sec:Bases}). Consequently, while there can be many different valuatively independent bases, the tropical theta functions and the valuative cost function are \emph{canonical objects} associated to the family. 

Moreover, the closure of the set of tropical theta functions in the space of continuous functions on $\operatorname{Sk}(X,\pi)$, which we will denote by $B$, is another invariant of the family $\pi:(X,L)\rightarrow \mathbb D^*$. While the valuative cost function is defined on $\operatorname{Sk}(X,\pi)\times B$, a priori, very little is known about this set. Given the conjectural picture explained above, it is natural to expect that $B$ is a polyhedral complex of dimension $d$. % 

The purpose of this paper is to provide an explicit construction of a sequence of bases $H^0(X,L^k), k\geq 1$ when $(X,L)$ is given by the Fermat family of cubic curves in $\mathbb P^2$, and use this to describe the tropical theta functions, $B$ and the valuative cost function explicitly (see statements~2, 3 and 4 of Theorem~\ref{thm:main} and Proposition~\ref{prop:CostFormula}). 

Before stating the theorem, we recall that a basis $\theta_1,\ldots,\theta_r$ of $H^0(X,L^k)$ is \emph{valuatively independent} if for any $n\in \op{Sk}(X,\pi)$ and $a_i \in \CC( \! (t) \! )$, we have 
\begin{equation}
\val_n((a_1\theta_1+\dots+a_r\theta_r)/\tau) = \min\{\val_t(a_i)+\val_n(\theta_i/\tau) \ : \ a_i \neq 0 \}. \label{eq:ValIndep}
\end{equation}
(Here, $\tau$ is a local trivialising section of $L^k$ and the notion of valuative independence is independent of its choice.)
\begin{theorem}\label{thm:main}
    Let $X$ be the Fermat family of cubic curves in $\mathbb P^2$, 
    $ x_0x_1x_2-t(x_0^3+x_1^3+x_2^3) = 0, $
    polarised by the restriction $L$ of $-K_{\mathbb P^2}$ and let $k\geq 1$. Then, $H^0(X,L^k)$ has a valuatively independent $\CC ( \! (t) \! )$-basis $\{\theta_m^k\}$ with the following properties.
    \begin{enumerate}
        \item For each $k\geq 1$, there exists a linear map $S_k \in GL(H^0(X,L^k), \CC( \! (t) \! ))$ such that 
        \[
        \theta^k_m = \sigma^k_m + tS_k(\sigma^k_m)
        \]
        where $\sigma^k_m$ denotes the (restriction of) the degree $k$ monomial section corresponding to $m\in \partial\Delta_k$. The coefficients of $S_k$ in the monomial basis $\{\sigma^k_m\}$ are formal power series in $t$ with integer coefficients.
        \item For each $k, \ell \geq 1$, and each $m \in \partial \Delta_k \subseteq \partial \Delta_{k\ell}$, we have 
        \[
        \frac{1}{k\ell}\val_n(\theta^{k\ell}_m) = \frac{1}{k}\val_n (\theta^k_m)
        \]
        for each $n\in \op{Sk}(X,\pi)$. In particular, the function $\hat\theta_m \in C^0(\op{Sk}(X,\pi),\RR)$ defined by $n \mapsto k^{-1}\val_n(\theta^k_m)$ (called the \emph{normalised tropical theta function associated to $m$}) does not depend on $k$.
    \item Let 
    \[
    B \subseteq C^0(\op{Sk}(X,\pi),\RR)
    \]
    denote the closure of the set of normalised tropical theta functions in the normed space $C^0(\op{Sk}(X,\pi),\RR)$. There is a homeomorphism  $\partial\Delta \longrightarrow B$
    taking $m \in \partial\Delta\cap\QQ^2$ to $\hat\theta_m \in B$.
    \item Letting $\hat\theta_m \in B \subseteq C^0(\op{Sk}(X,\pi),\RR)$ also denote the image of $m \in \partial\Delta$ (which may not necessarily be a rational point) under the homeomorphism in 3, we have that $\hat\theta_m$ satisfies the real-Monge-Ampère equation 
    $$\op{MA}(-\hat\theta_m)~=~9\delta_{F(m)},$$ 
    where $F: \partial \Delta \to \partial \Delta^\vee$ is given by
    \[
    F(m) = \frac{1}{3}\sum_{v\in\op{Vert(\Delta^\vee))}} (1 - \langle v, m\rangle) v \in \partial\Delta^\vee.
    \]
\end{enumerate}
\end{theorem}

Theorem~\ref{thm:main} establishes that $B$ in \eqref{eq:ValuativeCost} is given by the boundary of the (anti-canonical) moment polytope of $\mathbb P^2$ and valuatively independent sections of $H^0(X,L^k), k\geq 1$, are naturally identified (by their leading order terms) with its lattice points. Moreover, $\op{Sk}(X,\pi)$ carries a natural Hessian structure induced by the toric data of $\mathbb P^2$, which in turn is used to define the Monge-Ampère operator appearing in statement 4 of the theorem above (see Section~\ref{sec:CostFunction}, in particular Corollary~\ref{cor:MA}). 
\begin{remark}
    The monomial sections $\{\sigma^k_m\}$ of $H^0(\mathbb P^2,(-K_{\mathbb P^2})^k)$ also determine (by restriction) a basis of $H^0(X,L^k)$. However, this basis is not valuatively independent. The tropicalizations of the elements in this basis are not fundamental solutions to a real Monge-Ampère equation and, while the basis does induce a cost function which incidentally agrees with the cost function determined by the ambient toric manifolds in $\cite{HJMM,LiToric,AH}$, this cost function is defective in the sense that it does not always produce solutions to real Monge-Ampère equations (see Example~1--3 in \cite{AH}, Example~2.34 in \cite{LiToric}). A notable consequence of the theorem above is that valuative cost function is different from the cost function considered in \cite{HJMM,AH, LiToric}.  
\end{remark}

 As explained in \cite{LiCanonicalBasis}, constructing a basis which is \emph{locally} valuatively independent, i.e. satisfies \eqref{eq:ValIndep} for all $n$ in a small open set of $\op{Sk}(X,\pi)$, is a relatively straightforward task. The key difficulty is to produce a basis such that \eqref{eq:ValIndep} holds globally on $\op{Sk}(X,\pi)$. This problem has recently attracted considerable attention. For example, see \cite{BL, CMMM}. 
 
 We will now outline how Theorem~\ref{thm:main} is proved. A central idea in our method of proof is to use a canonical cost function $c:\op{Sk}(X,\pi)\times B \to \RR$ induced by the Hessian structure on $\operatorname{Sk}(X,\pi)$, and look for a basis satisfying~\eqref{eq:ValuativeCost}. The canonical cost function $c$, which was introduced in \cite{HO}, is explicitly computable using the monodromy data of $\op{Sk}(X,\pi)$ as an affine manifold (see Section~\ref{sec:CostFunction}). 

It provides, for each basis element, an \emph{a priori} indication of the vanishing orders along each strata of the central fiber, and thus a heuristic for how to construct such a basis. From the explicit expression for $c$, we can inductively construct the valuatively independent basis for $H^0(X,L^k)$ from valuatively independent bases for $H^0(X,L^r)$ with $r<k$ by adding correction terms to monomials of the form $\lambda\sigma \cdot S$ where $\sigma$ is a section given by an appropriate monomial term, $S$ is an appropriate member of the valuatively independent basis of $H^0(X,L^r)$ 
and $\lambda$ is an appropriate constant.

We expect some version Theorem~\ref{thm:main} to hold for general polarised generating families of Calabi-Yau manifolds. Before moving on, we will explain with some more detail what can be expected for the Dwork family in a toric Fano manifold $Y$ with moment polytope $\Delta \subset M_\mathbb R = M_Y\otimes \mathbb R$. First of all, for any vertex $v$ of its dual $\Delta^\vee\subset N_\mathbb R = N_Y\otimes \mathbb R$, the quotient map $\alpha_v: N_\mathbb R \rightarrow N_\mathbb R/v\mathbb R$ is injective on the star of $v$ in $\partial \Delta^\vee$ and any two such maps $\alpha_v,\alpha_{v'}$ induce compatible integral affine structures on the open faces of $\op{St}(v)\cap\St(v')$. This means that given a choice of singular locus $\Sigma$ which separates each co-dimension one face in $\partial\Delta^\vee$ into components, one for each vertex, we get an integral affine structure and thus a real Monge-Ampère operator $\op{MA}$ on $\partial \Delta^\vee\setminus \Sigma$. 

\begin{conjecture}
    Let $Y$ be a toric Fano manifold with moment polytope $\Delta$ and $c:B\times \op{Sk}(X,\pi)\to\RR$ be the valuative cost function associated to the Dwork family $\pi : (X,L) \to \mathbb D^*$ in $Y\times \mathbb D^*$. Then, $H^0(X,L^k)$ has a valuatively independent $\CC ( \! (t) \! )$-basis $\{\theta_m^k\}$ such that statements~1, 2 and 3 of Theorem~\ref{thm:main} hold verbatim. Moreover,
    the following version of statement 4 holds: There is a function $F: \partial\Delta \to \partial \Delta^\vee$ such that for each $m\in \partial \Delta$, there is a choice of singular locus $\Sigma_m$ such that $F(m)\notin \Sigma_m$ and
    \[\op{MA}(-\hat\theta_m)~=~(L_t)^{\dim X_t}\delta_{F(m)}\]
    on $\partial\Delta^\vee\setminus \Sigma_m$.
\end{conjecture}

\noindent \textbf{Outline of the paper.} In Section 2, we obtain an ansatz of the valuative cost function from the framework of Hessian manifolds and affine $\RR$-bundles introduced in \cite{HO}. In Section 3, we extract the valuative information from the ansatz and inductively construct a basis whose valuations match with the cost function. 

\

\noindent \textbf{Acknowledgements.}
The authors would like to thank Rolf Andreasson, Harold Blum, Yang Li, Yuchen Liu,  Mattias Jonsson and Nick McCleerey for fruitful discussions and thoughtful input. 

\ 

\noindent \textbf{AI disclosure.}
The authors acknowledge the use of generative AI tools to modify the TikZ figures in the paper. No other part of this work was done with the help of AI tools.

\section{Cost function in the theory of Hessian manifolds}\label{sec:CostFunction}
In this section we will explain how the toric data $\Delta, \Delta^\vee$ of $\mathbb P^2$ determines dual Hessian structures on $\partial \Delta$ and $\partial \Delta^\vee$ and how this, following \cite{HO}, defines a canonical cost function on $\partial \Delta \times \partial \Delta^\vee$. 

We begin by fixing a choice of integral basis $e_1, e_2$ in $\RR^2$ such that 
\[\Delta = \op{conv}(-e_1-e_2, -e_1+2e_2,2e_1-e_2),\quad \Delta^\vee = \op{conv}(-e_1^\vee, e_1^\vee + e_2^\vee,-e_2^\vee). \]
For future reference, we denote
\[
m_0 = -e_1-e_2, m_1 = -e_1+2e_2, m_2 = 2e_1-e_2, \quad n_0 = -e_1^\vee, n_1 = e_1^\vee + e_2^\vee, n_2=-e_2^\vee.
\]

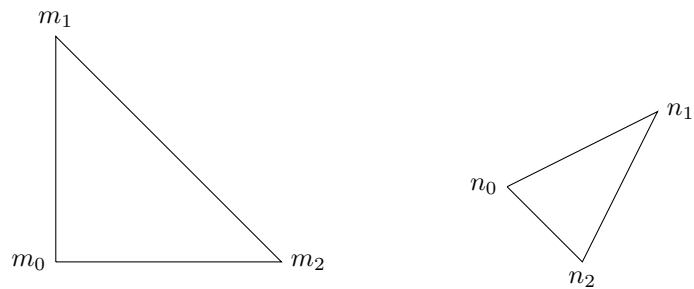
\begin{center}
    \begin{tikzpicture}
        \draw[thick]
        (-1,-1) node[left] {$m_0$}
        -- 
        (-1,2) node[above] {$m_1$}
        -- 
        (2,-1) node[right] {$m_2$}
        -- 
        (-1,-1);
        \draw[thick] 
        (5,0) node[left] {$n_0$}
        -- 
        (7,1) node[right] {$n_1$}
        -- 
        (6,-1) node[below] {$n_2$}
        -- 
        (5,0);
    \end{tikzpicture}
\end{center}
Note also that we have labelled the $m_k$ and $n_\ell$ in such a way that for $k=0,1,2$, 
we have 
\[
\langle m_k,n_k \rangle = \langle m_{k+1}, n_k \rangle = \langle m_k, n_{k-1}\rangle = 1, \quad \langle m_k, n_{k+1} \rangle = -2,
\]
where we take the subscripts modulo 3. We denote by $\Pi$ the fundamental group $\pi_1(\partial \Delta^\vee,n_0)$, and identify it with $\ZZ$ by taking the clockwise generator $\gamma$ to be positive.

The dual intersection complex of the central fibre $\{x_0x_1x_2=0\}$ is naturally identified with $\partial \Delta^\vee$. On the (open) star $U_k$ in $\partial\Delta^\vee$ of each vertex $n_k$, the function $m_{k+1}-m_k$ determines an affine coordinate, thus determining an affine structure on $\partial \Delta^\vee$. One way to describe this affine structure is as follows. The map $\pi:\RR \to \partial \Delta^\vee$ given for $\tau^\vee\in[k,k+1)$ by 
\begin{equation}\label{eqn:univcoverformula}
\pi(\tau^\vee) = (k+1-\tau^\vee)n_k + (\tau^\vee-k)n_{k+1}
\end{equation}
(where the subscripts are taken modulo 3) is a universal covering map, and its various local inverses give affine coordinates around each point of $\partial \Delta^\vee$. For example, a local inverse of $\pi$ is given by 
\[
\varphi_k: U_k \to (k-1,k+1), n\mapsto k + \frac{1}{3}\langle m_{k+1}-m_k,n\rangle.
\]
Note that we have made choices for the basepoints of the covering map in such a way that $\varphi_0(n) = \varphi_1(n)$, $\varphi_1(n) = \varphi_2(n)$ but $\varphi_2(n) = \varphi_0(n) + 3$ whenever both sides are defined.

\begin{definition}
An affine $\mathbb R$-bundle $L$ over an affine manifold $M$ is a principal $(\mathbb R,+)$-bundle with affine transition functions. Equivalently, $L$ is an affine manifold equipped with an affine map $L\rightarrow M$ with fibers isomorphic to $\mathbb R$ admitting local trivializations $\{U_i\times \mathbb R,p_i)$ such that the transition functions $p_i\circ p_j^{-1}:(U_j\cap U_i)\times \mathbb R\rightarrow U_i\times \mathbb R$ are of the form 
$$ (n,y)\rightarrow (n,y+\alpha_{ij}(x)) $$
for some affine transition function $\alpha_{ij}:U_j\cap U_i\rightarrow \mathbb R$.
\end{definition}

We will now describe the affine $\RR$-bundle $L$ on $M=\partial\Delta^\vee$ which admits convex sections. This determines $M= \partial\Delta^\vee$ as a Hessian manifold. 
\begin{lemma}
    The map $n\mapsto -\langle m_k,n\rangle$ is piecewise affine and convex on $U_k$.
\end{lemma}
\begin{proof}
We directly verify this using affine coordinates furnished by (the local inverses of) $\pi$. Indeed, for $\tau^\vee\in (k-1,k)$, we have 
\[
-\langle m_k, \pi(\tau^\vee)\rangle = -\langle m_k, (k-\tau^\vee)n_{k-1}+(\tau^\vee-k+1)n_k\rangle = -1,
\]
and for $\tau^\vee\in [k,k+1)$, we have 
\[
-\langle m_k, \pi(\tau^\vee)\rangle = -\langle m_k, (k+1-\tau^\vee)n_{k}+(\tau^\vee-k)n_{k+1}\rangle = 3\tau^\vee-3k-1.
\]
This shows that on $U_k$, the map $n\mapsto -\langle m_k, n\rangle $ is given by 
\[n\mapsto \max\{-1,3\varphi_k(n)-3k-1\}\]
which is piecewise affine and convex. 
\end{proof}

The affine $\RR$-bundle $L$ on $\partial \Delta^\vee$ is specified by the requirement that the maps $n\mapsto -\langle m_k,n\rangle$ glue to give a continuous, convex section. This gives a natural trivialisation of $L$ over $U_k$, say $L_k = L\vert_{U_k} \cong U_k \times \RR$ such that over $U_{k,k+1} = U_k\cap U_{k+1}$, the transition function from $L_k$ to $L_{k+1}$ is given by the affine function $n\mapsto -\langle m_{k+1} - m_{k}, n\rangle= -3\varphi_k(n) +3k$. With these transition functions, it is clear that the local functions $n\mapsto -\langle m_k, n \rangle$ glue together to give a global continuous section $\Phi_0$ of $L$ which is convex and piecewise affine.

Following \cite{HO}, we will now explain how the pair $(M,L)$ gives rise to a dual affine manifold equipped with an affine $\mathbb R$-bundle $(M^*,L^*)$ and a \emph{canonical pairing} $[\cdot,\cdot]$ which is a section of $L\boxplus -L^*\rightarrow M\times M^*$. 

First of all, let $\op{Aff}(L,n_0)$ be the space of germs of affine sections of $L$ at $n_0$. Let $\pi^* L$ be the pullback of $L$ to the universal covering space of $M$. The fundamental group $\Pi$ acts on $\pi^* L$ by deck transformation and, identifying $\op{Aff}(L,n_0)$ by the germ of affine functions of $\pi^*L$ at a point above $n_0$ in the universal cover, we may write the action of $\Pi$ on $\op{Aff}(L,n_0)$ as
$$ \gamma \cdot s = \gamma \circ s \circ \gamma^{-1}. $$
\begin{definition}\cite{HO}
    The affine $\mathbb R$-bundle $L^\star\rightarrow M^\star$ dual to $L\rightarrow M$ is
    $$ L^\star= \op{Aff}(L,n_0)/\Pi \rightarrow M^\star $$
    where the dual affine manifold $M^*$ is attained from this by taking the quotient $L^*/\mathbb R$, where $\mathbb R$ acts additively on the elements in $\op{Aff}(L,n_0)$. 
\end{definition}
\begin{definition}
    The \emph{cost pairing} $[\cdot,\cdot]$ is the section of  $L\boxplus-L^\star \to M\times M^\star$ defined by
\[
[s,x] = \sup_{\gamma \in \Pi}(\gamma \cdot s)(x) - s.
\]
and, fixing trivialising sections $\Phi$ and $\Psi$ of $L$ and $L^*$ respectively, the \emph{cost function} $c:M^\star\times M \to \RR$ is given by
$$ c(s,x) := -[s,x]+\Phi(x)-\Psi(s). $$
\end{definition}

\begin{lemma}\label{lemma:AffineIso}
    There are affine isomorphisms 
    \[
L \cong \frac{\RR\times\RR}{\langle (\tau^\vee,a)\mapsto (\tau^\vee-3,a-9\tau^\vee+9)\rangle} \longrightarrow M \cong \frac{\RR}{\langle \tau^\vee\mapsto \tau^\vee-3\rangle},
\]
and 
\[
L^\star \cong \frac{\RR\times\RR}{\langle (\tau,b)\mapsto (\tau-9,b+3\tau-18)\rangle} \longrightarrow M^\star \cong \frac{\RR}{\langle \tau\mapsto \tau-9\rangle}
\]
which are compatible with the action of $\Pi\cong\ZZ$ on the universal covers  of $M,M^\star,L,L^\star$.
\end{lemma}
\begin{proof}
Let $\op{Aff}(L,n_0)=\op{Aff}(L_0,n_0)$ be the space of germs of affine sections of $L$ at $n_0$. In the local affine coordinate $\varphi_0(n)=\frac{1}{3}\langle m_1 - m_0,n\rangle$, an arbitrary element $s_0\in \op{Aff}(L,n_0)$ has a representation as an affine function of $\varphi_0$, that is, $s_0(n) = (n,\tau\varphi_0(n)+b)$ for a uniquely determined pair $(\tau,b)\in \RR^2$. The section $s_0$ agrees with the affine section $s_1$ of $L|_{U_1}$ given by $s_1(n)=(n,(\tau-3)\varphi_1(n)+b)$, which in turn agrees with the section $s_2$ of $L|_{U_2}$ given by $s_2(n)=(n,(\tau-6)\varphi_2(n)+b+3)$. This last section $s_2$ agrees with the affine section $s_0^\prime$ of $L_0$ given by $s_0^\prime(n) = (n,(\tau-9)(\varphi_0(n)+3)+b+9)$. Thus, the action  of the positive generator $\gamma$ of $\Pi=\ZZ$ on $\op{Aff}(L,n_0)$ is to send the section given by the pair  $(\tau, b)$ to one given by the pair $(\tau - 9, b+3\tau-18)$. This proves 
\[L^\star \cong \frac{\RR\times\RR}{\langle (\tau,b)\mapsto (\tau-9,b+3\tau-18)\rangle}
\]
and since $M^\star = L^\star/\RR$, it also proves 
\[
M^\star \cong \frac{\RR}{\langle \tau \mapsto \tau -9\rangle}.
\]

Denoting by $s_{(\tau,b)}$ the germ determined by the pair $(\tau, b)\in \RR^2$, we therefore have 
\[
\gamma \cdot s_{(\tau,b)} = s_{(\tau - 9, b+3\tau - 18)}.
\]
With this in hand, we can determine the action of $\gamma\in \Pi$ on $\pi^* L \to \RR$ by using the formula 
\[
\gamma \cdot s = \gamma\circ s \circ \gamma^{-1}.
\]
As a smooth manifold, $\pi^*L \cong \RR\times \RR$ and the above formula for the action of $\gamma$ gives us 
\[
\gamma\cdot s_{(\tau,b)}(\tau^\vee) = s_{(\tau-9,b+3\tau-18)}(\tau^\vee) = (\tau^\vee,(\tau-9)\tau^\vee + b+3\tau -18).
\]
Since the action of $\gamma$ on $(t,a)\in \pi^*L\cong \RR\times\RR$ is affine and respects the projection $\pi^*L\to \RR$, we must have 
\[
\gamma(\tau^\vee,a)= (\tau^\vee\pm 3, A\tau^\vee + Ba + C)
\]
for some $A, B, C \in \RR$. But then we get 
\[
\gamma(s_{(\tau,b)}(\gamma^{-1}(\tau^\vee)))=(\tau^\vee\mp 3\pm 3, A(\tau^\vee\pm 3)+B(\tau(\tau^\vee \pm 3)+b)+C)
\]
and equating the two expressions gives $\gamma(\tau^\vee)=\tau^\vee-3$ and $\gamma(\tau^\vee,a)=(\tau^\vee-3,a-9\tau^\vee+9)$. This proves 
  \[
L \cong \frac{\RR\times\RR}{\langle (\tau^\vee,a)\mapsto (\tau^\vee-3,a-9\tau^\vee+9)\rangle} \longrightarrow M \cong \frac{\RR}{\langle \tau^\vee\mapsto \tau^\vee-3\rangle}.
\]
\end{proof}

The manifold $M^\star$ in this presentation is identified with $\partial \Delta$ via the map $p:\RR\to \partial\Delta$ given by 
\begin{equation}\label{eqn:univcoverformulavee}
p(\tau) = \left(k+1-\frac{\tau}{3}\right)m_k + \left(\frac{\tau}{3}-k\right)m_{k+1} \textrm{ for } \tau \in[3k,3(k+1)).
\end{equation}

It is easy to check that $L$ and $-L^\star$ admit convex sections. 
\begin{prop}\label{prop:CostFormula}
    In terms of the affine isomorphisms furnished by Lemma~\ref{lemma:AffineIso}, the cost function $c:M \times M^\star\to \RR$ is (induced from the map $\RR\times\RR \to \RR$) given by
    \begin{equation}\label{eqn:cost}
c(\tau^\vee,\tau) = 3(k+1)\tau^\vee + \ell \tau-(\tau -9m)(\tau^\vee+3m) - \frac{3k(k+1)+3\ell(\ell+1)+9m(3m-1)}{2}
\end{equation}
where
\[
k= \lfloor \tau^\vee \rfloor, \quad \ell = \left\lfloor \frac{\tau}{3}\right\rfloor, \quad m =\left \lfloor \frac{\tau-3\tau^\vee}{9}+\frac{1}{3}\right\rfloor.
\]
\end{prop}

\begin{proof}
Under the isomorphisms furnished by the Lemma, the pullback $\pi^*\Phi_0$ of the piecewise affine convex section $\Phi_0$ of $L$  is given by 
\[
\pi^*\Phi_0(\tau^\vee)=\left(\tau^\vee,\sup_{k\in\ZZ}\left(3k\tau^\vee-1-\frac{3k(k-1)}{2}\right)\right) = \left(\tau^\vee,3(\lfloor \tau^\vee \rfloor +1)\tau^\vee-1-\frac{3\lfloor \tau^\vee \rfloor(\lfloor \tau^\vee \rfloor+1)}{2}\right).
\]
The section of $[\cdot,\cdot]$ of $L\boxplus-L^\star \to M\times M^\star$ under this isomorphism is induced by the pairing
\[
[\tau^\vee,\tau] = \sup_{k\in\ZZ}(\gamma^k \cdot s_{(\tau,0)})(\tau^\vee) - s_{(\tau,0)}.
\]
(Note that we could have used $s_{(\tau,b)}$ for any $b\in \RR$ in the above expression without affecting its value, since the definition of the $\RR$-action on $L\boxplus-L^\star$ means that $(\tau^\vee,a)-s_{(\tau,b)}$ and $(\tau^\vee,a-b)-s_{(\tau,0)}$ represent the same points in the fibre over $(t,\tau)$.) In order to write this explicitly, we need to determine the supremum
\begin{align*}
    \sup_{k\in\ZZ}(\gamma^k \cdot s_{(\tau,0)})(\tau^\vee) &= \left(\tau^\vee,\tau \tau^\vee +\sup_{k\in\ZZ}\left(3k\tau-9k\tau^\vee-18k-27\frac{k(k-1)}{2}\right)\right)\\
    &=\left(\tau^\vee,\sup_{k\in\ZZ} \left( \! (\tau-9k)(\tau^\vee+3k)+27k^2-18k-27\frac{k(k-1)}{2}\right)\right)
\end{align*}
for any given value of $(\tau^\vee,\tau)$. Now 

\begin{align*}
    3k\tau-9k\tau^\vee-18k-27\frac{k(k-1)}{2}&=\frac{9k}{2}\left(2(\tau/3 - \tau^\vee)-4-3(k-1)\right)\\
    &=\frac{9k}{2}\left(2(\tau/3 - \tau^\vee)-1-3k\right)\\
    &= \frac{9k}{2}(6r-3k)\\
    &= \frac{27}{2}(r^2-(k-r)^2).
\end{align*}
where for convenience we have written $r=\frac{1}{9}(\tau-3\tau^\vee)-\frac{1}{6}$. This last expression attains its largest value when $|k-r|$ attains its smallest value. If $\{r\} = r-\lfloor r \rfloor \leq\frac{1}{2}$, then this happens when $k=\lfloor r \rfloor$ and if $\{r\} \geq \frac{1}{2}$, this happens for $k=\lfloor r \rfloor +1$. Thus, we get 
\[
[\tau^\vee,\tau] =
\begin{cases}
    (\tau^\vee, \tau \tau^\vee + \frac{27}{2}(r^2-\{r\}^2))-s_{(\tau,0)}, & \textrm{ if } \{r\}\leq \frac{1}{2},\\
    \\
    (\tau^\vee, \tau \tau^\vee + \frac{27}{2}(r^2-(1-\{r\})^2))-s_{(\tau,0)}, & \textrm{ if } \{r\}\geq \frac{1}{2}.
\end{cases}
\]
Note that the two expressions clearly coincide when $\{r\} = \frac{1}{2}$. When $r\in[k-1/2,k+1/2]$ for $k\in \ZZ$, we can easily check that the two expressions both reduce to $\tau \tau^\vee + \frac{27}{2}(r^2-(r-k)^2)$ (with both boundary points giving the same value $\frac{27}{2}(r^2+\frac{1}{4})$). Now $r\in [k-1/2,k+1/2)$ if and only if $k=\left\lfloor \frac{\tau-3\tau^\vee}{9}+\frac{1}{3}\right\rfloor$ and so we obtain
\begin{equation}\label{eqn:costpairing}
[\tau^\vee,\tau]=\left(\tau^\vee,(\tau-9k)(\tau^\vee+3k)+\frac{9k(3k-1)}{2}\right) -s_{(\tau,0)} \textrm{ where } k=\left\lfloor \frac{\tau-3\tau^\vee}{9}+\frac{1}{3}\right\rfloor.
\end{equation}
Given a continuous section $\Phi$ of $L$, its \emph{Legendre transform} $\Phi^\star$ is a section (induced by the function) defined by the formula 
\[
\Psi^\star(\tau) = \sup_{\tau^\vee\in \RR} \left([\tau^\vee,\tau]-\Psi(\tau^\vee)\right).
\]
and is a continuous section of $-L^\star$ by construction. We shall use $\Phi_0$ and $\Psi = \Phi_0^\star$ to write down the expression for $c$, where the section $\Phi_0$ is the convex section which on $U_i$ is given by $n\mapsto -\langle m_i,n\rangle$. We can write down explicitly:
\[
\Phi_0^\star(\tau)=\sup_{t\in \RR} \left([
\tau^\vee,\tau]-\Phi_0(
\tau^\vee
)\right) = \left(\tau, \lfloor \tau /3 \rfloor 
\tau^\vee+ 1-\frac{3\lfloor \tau /3 \rfloor(\lfloor \tau /3 \rfloor+1)}{2} \right).
\]
Finally, the cost function $c:M\times M^\star\to \RR$ is induced by the function (which we still denote by $c$) given by $c(\tau^\vee,\tau)=-[\tau^\vee,\tau]+\Phi_0(\tau^\vee)+\Phi_0^\star(\tau)$ which we can write explicitly as 
\begin{equation}\label{eq:Cost}
c(\tau^\vee,\tau) = 3(k+1)\tau^\vee + \ell \tau-(\tau -9m)(\tau^\vee+3m) - \frac{3k(k+1)+3\ell(\ell+1)+9m(3m-1)}{2}
\end{equation}
where
\[
k= \lfloor \tau^\vee \rfloor, \quad \ell = \left\lfloor \frac{\tau}{3}\right\rfloor, \quad m =\left \lfloor \frac{\tau-3\tau^\vee}{9}+\frac{1}{3}\right\rfloor.
\]
\end{proof}

We will now make the observation regarding the Monge-Ampère masses of tropical theta functions that will give the fourth bullet point of Theorem~\ref{thm:main}. Before that we note that the literature contains three different, but equivalent, settings for Monge-Ampère equations on $\Sk(X,\pi)$. In $\cite{LiToric, AH}$ a \emph{function} $\Psi$ on $\Sk(X,\pi)\cong \partial \Delta^\vee$ satisfies the Monge-Ampère equation 
\begin{equation}\label{eqn:MA}\MA(\Psi) = \mu,
\end{equation}
for a measure $\mu$ if for each $i\in \{0,1,2\}$, $\Psi-\langle m_i, \cdot\rangle$ satisfies the same Monge-Ampère equation in affine coordinates on $U_i$ (see \cite[Definition 3]{AH} for details). Similarly, \cite{ChengYau} fixes a \emph{global} reference form $g_{ij}$ which can locally be written as the second derivative of a convex function, and says that $\Psi$ solves the Monge-Ampère equation \eqref{eqn:MA} above if $\det(\Psi_{ij}+g_{ij})=\mu$. Finally, $\cite{HO}$ says that a \emph{section} of an affine line bundle over an affine manifold satisfies the Monge-Ampère above if its local trivialization satisfies the same Monge-Ampère equation in affine coordinates. These three perspectives are equivalent in the sense that if $s_0$ is the section of the line bundle $L\rightarrow M$ which in the local trivialization on $U_i$ is given by $-\langle \cdot , m_i\rangle$, $g_{ij}$ is the second derivative of $s_0$ (in the distributional sense), $s$ is another section of $L\rightarrow M$ and $\Psi=s-s_0$, then the Monge-Ampère of $s$ in the sense of \cite{HO} coincides with the Monge-Ampère of $\Psi$ in the sense of \cite{ChengYau} and \cite{LiToric,AH}. 
\begin{corollary}\label{cor:MA}
    Let $m\in \partial\Delta$. Then, $[\cdot,m]$ is a convex section of $L\rightarrow M$ whose Monge-Ampère mass (in the sense of \cite{HO}) is $9\delta_{F(m)}$, where
    \[F(m) = \frac{1}{3}\sum_{i=0}^2(1-\langle m, n_i\rangle)n_i.
    \]
    Equivalently, $-c(m,\cdot)$ has Monge-Ampère measure (in the sense of \cite{LiToric,AH}) equal to $9\delta_{F(m)}$.
\end{corollary}
\begin{proof}
    Let $m \in \partial\Delta$ be given. Without loss of generality, we can assume that $m \in [m_0,m_1]$ belongs to the edge spanned by $m_0,m_1$ by symmetry. In the formula \eqref{eqn:costpairing} for the cost pairing $[\tau^\vee,\tau]$, this amounts to fixing $\tau \in [0,3]$ (corresponding to $m = \frac{\tau}{3}m_1+(1-\frac{\tau}{3})m_0$), so the Monge-Ampère measure associated to $\tau^\vee\mapsto [\tau^\vee,\tau]$ is simply given by the ordinary Monge-Ampère measure of the convex function 
    \[
    \tau^\vee\mapsto (\tau-9k)(\tau^\vee+3k)+\frac{9k(3k-1)}{2}, \quad k = \left\lfloor \frac{\tau-3\tau^\vee}{9}+\frac{1}{3}\right\rfloor
    \]
    on the interval $[0,3]$. This piecewise affine function changes slope at $\tau^\vee_0= \frac{\tau}{3}+1$ and its change in slope is equal to 9. Thus, we have 
    \[
    \op{MA}([\cdot,\tau]) = 9 \delta_{\tau^\vee_0}.
    \]
    By the identification of the universal cover given by \eqref{eqn:univcoverformula}, we have 
    \[
    \pi(\tau^\vee_0) = \left(1 + 1 - \frac{\tau}{3}-1\right)n_1 + \left(\frac{\tau}{3} + 1 - 1\right)n_2 = \left(1-\frac{\tau}{3}\right)n_1 + \frac{\tau}{3}n_2.  
    \]
    Now note that (since $m = \frac{\tau}{3}m_1+(1-\frac{\tau}{3})m_0$), we have 
    \[
    \langle m, n_0\rangle = 1, \quad \langle m, n_1\rangle = \tau - 2, \quad \langle m, n_2\rangle = 1-\tau.   \]
    Thus, 
    \[
    \pi(\tau^\vee_0) = \frac{1}{3}\sum_{i=0}^2(1-\langle m, n_i\rangle)n_i = F(m)
    \]
    and the conclusion follows.
\end{proof}

\section{Valuatively independent bases of sections}\label{sec:Bases}
Let us now return to the hypersurface $X = V(t(x_0^3+x_1^3+x_2^3)-x_0x_1x_2=0) \subseteq \PP^2\times\CC$. We may view this as a projective variety $X_K$ over $K= \CC( \! (t)\! )$ with an embedding in $\PP^2_{K}$ given by a section of $L_K=\mathcal O_{\PP^2_K}(3)$. The dual intersection complex of the central fibre $\{x_0x_1x_2=0\}$ is the boundary of a two-dimensional simplex. We can identify it with $M = \partial\Delta^\vee$ by identifying the vertex corresponding to the component $\{x_i=0\}$ with $n_i$. With this identification, a point $n =(1-r)n_i+rn_{i+1}$ corresponds to the valuation on $\CC(X)$ given by
\[\val_n\left(\sum_{p,q} c_{p,q} \left(\frac{x_i}{x_{i-1}}\right)^{p}\left(\frac{x_{i+1}}{x_{i-1}}\right)^q\right) = \min \left\{(1-r)p+rq \ : \ c_{p,q}\neq 0\right\}.\]
Recall that a collection $\theta_1,\dots,\theta_r \in H^0(X,L^\ell)$ is called \emph{valuatively independent} if for any $n\in [n_j,n_{j+1}]\subset  \partial\Delta^\vee$, and $a_i\in K$, we have 
\[
\val_n((a_1\theta_1+\dots+a_r\theta_r)/\tau) = \min\{\val_t(a_i)+\val_n(\theta_i/\tau) \ : \ a_i \neq 0 \}
\]
where $\tau=\tau(n)$ is any section of $L^\ell$ which is holomorphic and non-vanishing in a neighbourhood of the point $\{x_j = x_{j+1}=0\}$. (Note that the value of the expression $\val_n(s/\tau)$ does not depend on the choice of $\tau$ since $\val_n(\tau^\prime/\tau) = 0$ for any two $\tau^\prime,\tau$ which are non-nanishing and holomorphic in a neighbourhood of $\{x_i=x_{i+1}=0\}$. We will therefore simply write $\val_n(s)$ for $\val_n(s/\tau)$.)

In this section, we wish to prove Theorem \ref{thm:main} using the cost ansatz \eqref{eqn:cost}.
Toward this end, we first prove the following. 
\begin{prop}\label{prop:HessianCostisValuative} Fix $\ell \in\QQ_{>0}$ such that $L^\ell$ is a line bundle, and suppose we have a basis $s^\ell_m, m\in\partial\Delta_\ell$ of $H^0(X_K,L^\ell_K)$ such that 
\[\ell^{-1}\val_n(s^\ell_m)= c(m,n)
\]
for all $n\in\partial\Delta^\vee$ and $m\in\partial\Delta_\ell$. Then, the basis $s^\ell_m,m\in \partial\Delta_\ell$ is valuatively independent. 
\end{prop}
In other words, if the cost \eqref{eqn:cost} is realised valuatively from a basis, then this basis must be valuatively independent. 
\begin{proof}[Proof of Proposition~\ref{prop:HessianCostisValuative}]
    Let $\tau \in[0,3]$ be fixed. Then, from the explicit expression for the cost function $c$, we have that the function $x\mapsto c(\tau,x)$ satisfies
    \[
    c(\tau,x) = 
    \begin{cases}
        (3-\tau)x & \textrm{ if } x\in[0,1],\\
        \\
        (6-\tau)x + C_1 & \textrm{ if } x\in [1,1+\frac{\tau}{3}],\\
        \\
        -(\tau+3)x + C_2 & \textrm{ if } x\in[1+\frac{\tau}{3},2],\\
        \\
        -\tau x +C_3& \textrm{ if } x\in[2,3],
    \end{cases}
    \]
    where $C_1,C_2,C_3$ are constants uniquely determined by the requirement that $x\mapsto c(x,\tau)$ be continuous. Thus, $x\mapsto c(\tau,x)$ has the following slopes.
    
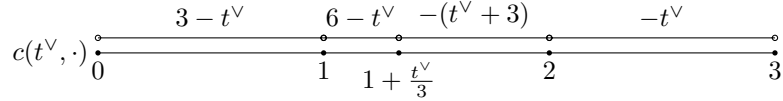
\begin{center}
\begin{tikzpicture}[x=3cm, y=0.8cm, >=stealth]

  % value of tau used for the picture (any value in [0,3])
  \pgfmathsetmacro{\tv}{1}
  \pgfmathsetmacro{\xm}{1+\tv/3}              % x-coordinate of middle breakpoint
  \pgfmathsetmacro{\ya}{3-\tv}                % c(tau,1)
  \pgfmathsetmacro{\ym}{3+\tv-\tv*\tv/3}      % c(tau,1+tau/3)

  % axes
  \draw[->] (0,0) -- (3.25,0) node[right] {$x$};
  \draw[->] (0,0) -- (0,4.3) node[above] {$c(\tau,x)$};

  % dashed guides down to the x-axis
  \draw[dashed, gray] (1,0)     -- (1,\ya);
  \draw[dashed, gray] (\xm,0)   -- (\xm,\ym);
  \draw[dashed, gray] (2,0)     -- (2,\tv);

  % the graph
  \draw[thick] (0,0) -- (1,\ya) -- (\xm,\ym) -- (2,\tv) -- (3,0);

  % points on the graph
  \foreach \p in {(0,0), (1,\ya), (\xm,\ym), (2,\tv), (3,0)}
    \fill \p circle (1.5pt);

  % ticks on the x-axis
  \foreach \x/\lab in {0/0, 1/1, 2/2, 3/3}
    \draw (\x,0.1) -- (\x,-0.1) node[below] {$\lab$};
  \draw (\xm,0.1) -- (\xm,-0.1) node[below] {$1+\frac{\tau}{3}$};

  % slope labels
  \node[above left]  at (0.5,{\ya/2})          {$3-\tau$};
  \node[left]        at ({(1+\xm)/2},{(\ya+\ym)/2}) {$6-\tau$};
  \node[above right] at ({(\xm+2)/2},{(\ym+\tv)/2}) {$-(\tau+3)$};
  \node[above right] at (2.5,{\tv/2})          {$-\tau$};

\end{tikzpicture}
\end{center}
Now, let $\ell \in \QQ_{>0}$ be such that $L^\ell_K$ is a line bundle. Then, $\ell = \frac{k}{3}$ for some positive integer $k$. If $m = \frac{a}{k}m_0+\frac{b}{k}m_1$ with $a,b\geq 0,a+b=k$, then $p(\tau) = m$ for $\tau = \frac{3b}{a+b}$. By hypothesis, we have $\val_n(s^\ell_m) = \ell c(m,n)$, whence we see that the function $n\mapsto \val_n(s^\ell_m)$ has the following slopes.

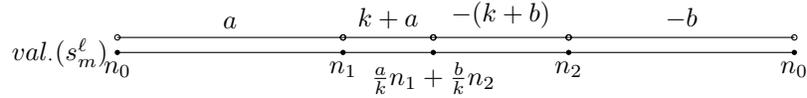
\begin{center}
\begin{tikzpicture}[x=3cm, y=0.5cm, >=stealth]

  % parameters: a,b >= 0, k = a+b
  \pgfmathsetmacro{\aa}{3}
  \pgfmathsetmacro{\bb}{2}
  \pgfmathsetmacro{\kk}{\aa+\bb}
  \pgfmathsetmacro{\xm}{1+\bb/\kk}               % position of (a/k)n_1+(b/k)n_2
  \pgfmathsetmacro{\yone}{\aa}                   % value at n_1
  \pgfmathsetmacro{\ym}{\kk+\aa*\bb/\kk}         % value at the middle breakpoint
  \pgfmathsetmacro{\ytwo}{\bb}                   % value at n_2

  % axes: the horizontal axis is the closed boundary, so no arrow
  \draw (0,0) -- (3,0) node[right] {$\partial\Delta^\vee$};
  \draw[->] (0,0) -- (0,7.2) node[above] {$\val_{(\cdot)}(s^\ell_m)$};

  % dashed guides down to the horizontal axis
  \draw[dashed, gray] (1,0)    -- (1,\yone);
  \draw[dashed, gray] (\xm,0)  -- (\xm,\ym);
  \draw[dashed, gray] (2,0)    -- (2,\ytwo);

  % the graph
  \draw[thick] (0,0) -- (1,\yone) -- (\xm,\ym) -- (2,\ytwo) -- (3,0);

  % points on the graph
  \fill (0,0)       circle (1.5pt);
  \fill (1,\yone)   circle (1.5pt);
  \fill (\xm,\ym)   circle (1.5pt);
  \fill (2,\ytwo)   circle (1.5pt);
  \fill (3,0)       circle (1.5pt);

  % vertices on the horizontal axis
  \foreach \x/\lab in {0/n_0, 1/n_1, 2/n_2, 3/n_0}
    \draw (\x,0.15) -- (\x,-0.15) node[below] {$\lab$};

  % middle breakpoint, label lowered to avoid colliding with n_1
  \draw (\xm,0.15) -- (\xm,-0.15);
  \node[below, yshift=-0.2cm] at (\xm,0) {$\frac{a}{k}n_1+\frac{b}{k}n_2$};

  % slope labels
  \node[above left]  at (0.5,{\yone/2})                    {$a$};
  \node[left]        at ({(1+\xm)/2},{(\yone+\ym)/2})      {$k+a$};
  \node[above right] at ({(\xm+2)/2},{(\ym+\ytwo)/2})      {$-(k+b)$};
  \node[above right] at (2.5,{\ytwo/2})                    {$-b$};

\end{tikzpicture}
\end{center}
Since our set-up is symmetric, similar statements hold when $m\in [m_1,m_2]$ or $m\in[m_2,m_0]$. It follows that if $n=(1-r)n_i+rn_{i+1}$ with $r \in [0,1]\setminus\{0,\frac{1}{k},\frac{2}{k},\dots,\frac{k-1}{k},1\}$, then each of the functions $n \mapsto \val_n(s^\ell_m)$ is differentiable at $n$ and the collection of all their slopes at $n$ is precisely 
\[
\{0,\pm 1,\pm 2,\dots,\pm k, k+1,k+2,\dots, k+k-s, -(k+1),-(k+2),\dots,-(k+s-1)\}
\]
where $s$ is the positive integer such that $r\in (\frac{s-1}{k},\frac{s}{k})$. This set has cardinality $3k = |\partial \Delta_\ell|$, which shows that each member of the collection $s_m, m\in \partial\Delta_\ell$ has a distinct slope for $n = (1-r)n_i + rn_{i+1}$ as long as $r \in [0,1]\setminus\{0,\frac{1}{k},\frac{2}{k},\dots,\frac{k-1}{k},1\}$. Thus, for all but finitely many points $n\in \partial\Delta^\vee$, each section $s^\ell_m$ has a unique leading term in its Taylor expansion distinct from those of the other sections $s^\ell_{m^\prime}, m\neq m^\prime$. Therefore, the leading term in any $K$-linear combination must be given by taking the minimum over the leading terms of the summands since no cancellation is possible. Now since both sides of
\[
\val_n\left(\sum_{m\in\partial\Delta_\ell}a_m s^\ell_m\right) = \min_{m\in\partial\Delta_\ell}\{\val_t(a_m)+\val_n(s^\ell_m) \ : \ a_m \neq 0 \}
\]
are continuous in $n$, the equality must still remain true for the remaining finitely many points.
\end{proof}

Thus, in order to establish Theorem \ref{thm:main}, it is sufficient to construct a basis $\theta^\ell_m$ whose valuative cost is given by \eqref{eqn:cost}. We shall do this inductively, using the explicit expression~\eqref{eq:Cost} for the cost function $c$.

Fix $\ell$ such that $L^\ell$ is a line bundle. We shall first define $\theta^\ell_m$ for $m=\frac{p}{p+q}m_0+\frac{q}{p+q}m_1 \in[m_0,m_1]$, for all values of $p,q\geq 0, p+q =\ell$ as a polynomial expression in $x_0,x_1,x_2$ and $t$, homogeneous of degree $\ell$ in the $x_i$. Recall that $m=\frac{p}{p+q}m_0+\frac{q}{p+q}m_1 \in[m_0,m_1]$ corresponds to the $\tau = \frac{3q}{p+q}$, and from~\eqref{eq:Cost} we get that
\begin{equation}\label{eq:VanishingOrders}
    \frac{p+q}{3}c\left(\frac{3q}{p+q},r\right) = 
    \begin{cases}
        pr & \textrm{ if } r\in[0,1],\\
        \\
        (2p+q)r + C_1 & \textrm{ if } r\in [1,1+\frac{q}{p+q}],\\
        \\
        -(2q+p)r + C_2 & \textrm{ if } r\in[1+\frac{q}{p+q},2],\\
        \\
        -qr +C_3& \textrm{ if } r\in[2,3],
    \end{cases}
    \end{equation}
    for some constants $C_i$. 
We need to define sections $\theta^\ell_m$ whose Taylor expansions have vanishing orders prescribed by the various slopes of this function. We shall do this by induction on $p+q$, giving $\theta^\ell_m$ as a polynomial expression in $x_0,x_1,x_2, t$, homogeneous of degree $p+q$ in the $x_i$. To this end, write $S(p,q,x_0,x_1,x_2)$ for this expression for $\theta^\ell_m$. First of all, define the following base cases:
\[
S(1,0,x_0,x_1,x_2) = x_1, \quad S(0,1,x_0,x_1,x_2) = x_2, \quad S(1,1,x_0,x_1,x_2) = x_1x_2-tx_0^2.
\] 
Let us suppress notation for the moment and write $S(p,q)$ for $S(p,q,x_0,x_1,x_2)$. 
Define 
\begin{equation}\label{eq:BaseCases}
    S(p,0)=S(1,0)^p,\quad S(0,q)= S(0,1)^q, \quad S(p,p)=S(1,1)^p
\end{equation}
for all $p,q\geq 0$. Note that this defines $S(p,q,x_0,x_1,x_2)$ whenever $p+q = 0,1,2$. Using the fact that the points $m_0$ and $m_1$ correspond, respectively, to $\tau=0$ and $\tau=3$, it is quite straightforward to verify, using~\eqref{eq:Cost}, that 
\[
\val_n(S(p,0)) = \frac{p}{3}c(m_0,n), \quad \val_n(S(0,q))=\frac{q}{3}c(m_1,n).
\]
The following Lemma shows that the analogous statement is true also for $S(p,p)$. 
\begin{lemma}\label{lemma:EqualityCase}
    Let $p\geq1$. Then, the Taylor expansion of $S(p,p)$ in $x=x_1/x_0,y=x_2/x_0$ is of the form 
    \[
    S(p,p) = x^py^p\left(\left(\frac{y^3}{1+y^3}\right)^p + \left(\frac{x^3}{1+x^3}\right)^p + x^3y^3H_{p,p}(x^3,y^3)\right)
    \]
where every  nonzero monomial term in the power series $x^3y^3H_{p,q}(x^3,y^3)$ is dominated by $x^{3p}+y^{3p}$. Moreover, for $n\in\partial\Delta^\vee$, we have 
\[
\val_n(S(p,p))=\frac{2p}{3}c\left(\frac{1}{2}m_0+\frac{1}{2}m_1,n\right).
\]
\end{lemma}
Above and hereafter, we say that a power series $f(x,y)$ \emph{dominates} another power series $g(x,y)$ if $\val_n(f)\leq \val_n(g)$ for each $n\in[n_1,n_2]$. With this terminology, if $f_1$ dominates $g_1$ and $f_2$ dominates $g_2$, then $f_1\cdot f_2$ dominates $g_1\cdot g_2$, and if the coefficients of all leading monomial terms of $f_1$ and $f_2$ are positive, then $f_1+f_2$ dominates $g_1\pm g_2$. 
\begin{proof}[Proof Lemma~\ref{lemma:EqualityCase}]
Computing the Taylor expansion of $S(p,p)$ and re-arranging the infinite sums, we get
    \begin{align*}
        S(p,p)&=(x_1x_2-tx_0^2)^p \\
        &= \left(xy-\frac{xy}{1+x^3+y^3}\right)^p\\
        &=x^py^p\left(\frac{x^3+y^3}{1+x^3+y^3}\right)^p\\
        &=x^py^p(x^3+y^3)^p\left(\sum_{k=0}^\infty(-1)^k \binom{p-1+k}{k} (x^3+y^3)^k \right)\\
        &= x^py^p\left(x^{3p} \sum_{k=0}^\infty(-1)^k \binom{p-1+k}{k} x^{3k} + y^{3p}\sum_{k=0}^\infty(-1)^k \binom{p-1+k}{k} y^{3k} + x^3y^3 H_{p,p}(x^3,y^3)\right)\\
        &= x^py^p\left(\left(\frac{y^3}{1+y^3}\right)^p + \left(\frac{x^3}{1+x^3}\right)^p + x^3y^3H_{p,p}(x^3,y^3)\right),
    \end{align*}
    for some power series $H_{p,p}(x^3,y^3)$. (In the intermediate steps, we have used a well-known formula for the Taylor expansion of powers of the geometric series.) Now, every nonzero monomial term in $x^3y^3H_{p,p}(x^3,y^3)$ is of the form $\lambda x^{3\ell}y^{3s}$ with $\lambda \neq 0$ and $\ell,s\geq 1, \ell+s\geq 3p$. This is clearly dominated by $x^{3p}+y^{3p}$. Thus, $S(p,p)$ is dominated by $x^py^p(x^{3p}+y^{3p})$ and the coefficients of the terms $x^py^p\cdot x^{3p}, x^py^p\cdot y^{3p}$ in its Taylor expansion are equal to 1. 

     Comparing with \eqref{eq:VanishingOrders}, the above shows that functions $n\mapsto \val_n(S(1,1))$ and $n\mapsto \frac{2p}{3}c(\frac{1}{2}m_0+\frac{1}{2}m_1,n)$ have the same slopes for $n\in[n_1,n_2]$. In order to establish that 
    \[
    \val_n(S(1,1))=\frac{2p}{3}c(m,n)
    \]
    holds for all $n\in\partial\Delta^\vee$, we will show they also have the same slopes at all other points $n \in \partial\Delta^\vee$, and both functions are equal to $0$ at $n_0$. That $c(\frac{1}{2}m_0+\frac{1}{2}m_1,n_0)=0$ is readily verified from \eqref{eq:Cost}. Let $n\in[n_0,n_1]$. Then, the Taylor expansion of $S(p,p)$ in $\xi=x_0/x_2,\eta=x_1/x_0$ is given by 
    \[
    S(p,p)=\left(\eta-\frac{\xi^3\eta}{1+\xi^3+\eta^3}\right)^p=\eta^p\left(\frac{1+\eta^3}{1+\xi^3+\eta^3}\right)
    \]
    which is dominated by $\eta^p$ and the coefficient of $\eta^p$ in this Taylor expansion is equal to $1$. This shows that $\val_{n_0}(S(1,1))=0$ and the two functions $n\mapsto \val_n(S(1,1))$ and $n\mapsto \frac{2p}{3}c(\frac{1}{2}m_0+\frac{1}{2}m_1,n)$ have the same slopes for all $n\in[n_0,n_1]$. A similar argument shows they have the same slopes for $n\in[n_2,n_0]$.
\end{proof}

From now on, fix $a,b\geq 1$ and assume that  $S(p,q)$ has been defined for $p+q<a+b$ in such a way that 
\[
\val_n(S(p,q)) = \frac{p+q}{3}c\left(\frac{p}{p+q}m_0+\frac{q}{p+q}m_1,n\right).
\]
In other words, its slopes as a piecewise affine function on $\partial\Delta^\vee$ are as follows.
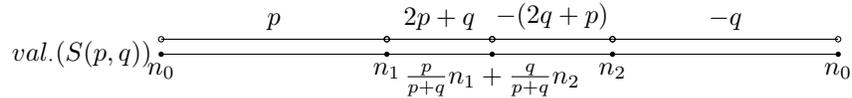
\begin{center}
\begin{tikzpicture}[x=3cm, y=0.5cm, >=stealth]

  % parameters: p,q >= 1
  \pgfmathsetmacro{\pp}{3}
  \pgfmathsetmacro{\qq}{2}
  \pgfmathsetmacro{\ss}{\pp+\qq}
  \pgfmathsetmacro{\xm}{1+\qq/\ss}               % position of p/(p+q) n_1 + q/(p+q) n_2
  \pgfmathsetmacro{\yone}{\pp}                   % value at n_1
  \pgfmathsetmacro{\ym}{\ss+\pp*\qq/\ss}         % value at the middle breakpoint
  \pgfmathsetmacro{\ytwo}{\qq}                   % value at n_2

  % axes: the horizontal axis is the closed boundary, so no arrow
  \draw (0,0) -- (3,0) node[right] {$\partial\Delta^\vee$};
  \draw[->] (0,0) -- (0,7.2) node[above] {$\val_{(\cdot)}(S(p,q))$};

  % dashed guides down to the horizontal axis
  \draw[dashed, gray] (1,0)    -- (1,\yone);
  \draw[dashed, gray] (\xm,0)  -- (\xm,\ym);
  \draw[dashed, gray] (2,0)    -- (2,\ytwo);

  % the graph
  \draw[thick] (0,0) -- (1,\yone) -- (\xm,\ym) -- (2,\ytwo) -- (3,0);

  % points on the graph
  \fill (0,0)       circle (1.5pt);
  \fill (1,\yone)   circle (1.5pt);
  \fill (\xm,\ym)   circle (1.5pt);
  \fill (2,\ytwo)   circle (1.5pt);
  \fill (3,0)       circle (1.5pt);

  % vertices on the horizontal axis
  \foreach \x/\lab in {0/n_0, 1/n_1, 2/n_2, 3/n_0}
    \draw (\x,0.15) -- (\x,-0.15) node[below] {$\lab$};

  % middle breakpoint, label lowered to avoid colliding with n_1
  \draw (\xm,0.15) -- (\xm,-0.15);
  \node[below, yshift=-0.2cm, xshift = 0.1cm] at (\xm,0) {$\frac{p}{p+q}n_1+\frac{q}{p+q}n_2$};

  % slope labels
  \node[above left]  at (0.5,{\yone/2})                    {$p$};
  \node[left]        at ({(1+\xm)/2},{(\yone+\ym)/2})      {$2p+q$};
  \node[above right] at ({(\xm+2)/2},{(\ym+\ytwo)/2})      {$-(2q+p)$};
  \node[above right] at (2.5,{\ytwo/2})                    {$-q$};

\end{tikzpicture}
\end{center}

This implies that, when $p,q\geq 1$, the Taylor expansion $S(p,q)$ in $x=x_1/x_0,y=x_2/x_0$ is of the form 
\[
S(p,q) = x^py^q(\lambda y^{3p}+\mu x^{3q}+ F(x) + G(y) + xy H(x,y))
\]
where $\lambda,\mu\neq 0$, $F(x) = O(x^{3q+3}), G(y)=O(y^{3p+3})$ and every nonzero monomial term in $xyH(x,y)$ is dominated by $x^{3q}+y^{3p}$ for every $n\in[n_1,n_2]$. For our purposes, we shall adopt a stronger inductive hypothesis. Namely, we shall assume moreover that whenever $p,q\geq 1$, the Taylor expansion of $S(p,q)$ in $x,y$ is of the form
\begin{equation}\label{eqn:CorrectTaylorSeries}
    S(p,q) = x^py^q\left(\left(\frac{y^3}{1+y^3}\right)^p + \left(\frac{x^3}{1+x^3}\right)^q + x^3y^3H_{p,q}(x^3,y^3)\right)
\end{equation}
where every (mixed) nonzero monomial term in $x^3y^3H_{p,q}(x^3,y^3)$ is dominated by $x^{3q}+y^{3p}$. Note that implicit in our inductive hypothesis is the assumption that $S(p,q)/x^py^q$ is a power series in $x^3$ and $y^3$. Note also that the base case $S(1,1)$ satisfies this stronger hypothesis by Lemma~\ref{lemma:EqualityCase}. 

Let $a,b\geq 0$ and assume, by induction, that $S(p,q)$ has been defined for all $p,q$ such that $p+q < a+b$ satisfying 
\[
\val_n(S(p,q)) = \frac{p+q}{3}c\left(\frac{p}{p+q}m_0+\frac{q}{p+q}m_1,n\right)
\]
for all $n\in\partial\Delta^\vee$ and satisfying~\eqref{eqn:CorrectTaylorSeries} if $p,q\geq 1$. If $a=b$ or one of $a,b$ is zero, Lemma~\ref{lemma:EqualityCase} and the remarks preceding it show that we can define $S(a,b)$ by ~\eqref{eq:BaseCases}, so assume that $a,b\geq 1, a\neq b$. Define the section $R_0(a,b)\in H^0(X,L^{(a+b)/3})$ given by 
\begin{equation}\label{eqn:R0ab}
R_0(a,b) = \begin{cases}
    x_1S(a-1,b)-t^2x_0S(a-2,b+1) & \textrm{ if } a>b\geq 1,\\
    \\
    x_2S(a,b-1)-t^2x_0S(a+1,b-2) & \textrm{ if } b > a\geq 1.
\end{cases}
\end{equation}
\begin{lemma}\label{lemma:R_0(a,b)}
    The Taylor expansion of $R_0(a,b)$ in $x,y$ is of the form
    \[
    R_0(a,b)=x^ay^b\left(\left(\frac{y^3}{1+y}\right)^{a}+\left(\frac{x^3}{1+x^3}\right)^{b}+x^3y^3G_{a,b}(x^3,y^3)\right)
    \]
    where $G_{a,b}(x^3,y^3)$ is some power series. Moreover, if $a>b\geq1$ (respectively, $b>a\geq 1$), then each nonzero monomial term in $x^3y^3G_{a,b}(x^3y^3)$ is dominated by $y^{3a-3}+x^{3b}$ (respectively, $y^{3a}+x^{3b-3}$).
\end{lemma}
\begin{rem}
    Note that we do not (and in fact cannot truthfully) make the stronger claim that the nonzero monomial terms in $x^3y^3G_{a,b}(x^3,y^3)$ are dominated by $y^{3a}+x^{3b}$.
\end{rem}
\begin{proof}[Proof of Lemma~\ref{lemma:R_0(a,b)}]
    Assume that $a>b\geq 1$. (The case whereby $b> a \geq 1$ is handled similarly by interchanging the roles of $x$ and $y$.) Then, by the inductive hypothesis~\eqref{eqn:CorrectTaylorSeries}, we have 
    \[
    S(a-1,b) = x^{a-1}y^b\left(\left(\frac{y^3}{1+y^3}\right)^{a-1} + \left(\frac{x^3}{1+x^3}\right)^{b} + x^3y^3H_{a-1,b}(x^3,y^3)\right)
    \]
    and 
    \[
    S(a-2,b+1) = x^{a-2}y^{b+1}\left(\left(\frac{y^3}{1+y^3}\right)^{a-2} + \left(\frac{x^3}{1+x^3}\right)^{b+1} + x^3y^3H_{a-2,b+1}(x^3,y^3)\right).
    \]
    Now observe that 
    \[
    t = \left(\frac{xy}{1+x^3+y^3}\right)=xy\left(\frac{x^3}{1+x^3}+\frac{1}{1+y^3}+x^3y^3F(x^3,y^3)\right) 
    \]
    for some power series $F(x^3,y^3)$. This in turn shows that 
    \[
    t^2=x^2y^2\left(\frac{1}{(1+y^3)^2}+x^3\tilde F(x^3,y^3)\right)
    \]
    for a power series $x^3\tilde F(x^3,y^3)$. Therefore we get
    \begin{align*}
    t^2x_0S(a-2,b+1) &=x^{a}y^{b+3}\left(\left(\frac{y^3}{1+y^3}\right)^{a-2} + \left(\frac{x^3}{1+x^3}\right)^{b+1} + x^3y^3H_{a-2,b+1}(x^3,y^3)\right)\left(\frac{1}{(1+y^3)^2}+x^3\tilde F(x^3,y^3)\right)\\
    &=x^{a}y^b\left(\frac{y^3}{(1+y^3)^2}\left(\frac{y^3}{1+y^3}\right)^{a-2}+x^3y^3G(x^3,y^3)\right) 
    \end{align*}
    for some power series $G(x^3,y^3)$. Thus, we conclude that 
    \begin{align*}
    R_0(a,b) &= x^{a}y^b\left(\left(\frac{y^3}{1+y^3}\right)^{a-1} +\left(\frac{x^3}{1+x^3}\right)^{b} -\frac{y^3}{(1+y^3)^2}\left(\frac{y^3}{1+y^3}\right)^{a-2}  + x^3y^3G_{a,b}(x^3,y^3)\right)\\
    &=x^ay^b\left(\left(\frac{y^3}{1+y}\right)^{a}+\left(\frac{x^3}{1+x^3}\right)^{b}+x^3y^3G_{a,b}(x^3,y^3)\right).
    \end{align*}
    for a power series $G_{a,b}(x^3,y^3)$. Now, by the inductive definition of $R_0(a,b)$, we see that each  nonzero monomial term in $R_0(a,b)$ is dominated by 
    \[x\cdot x^{a-1}y^b(y^{3a-3}+x^{3b})+x^2y^2\cdot x^{a-2}y^{b+1}(y^{3a-6}+x^{3b+3})
    \] 
    which itself is dominated by $x^ay^b(y^{3a-3}+x^{3b})$. This proves the Lemma.
\end{proof}
The above Lemma shows that $R_0(a,b)$ has the correct `pure' terms $x^ay^b(y^{3a}+x^{3b})$ but unfortunately, except for certain special cases, $R_0(a,b)$ will have several mixed terms that dominate these pure terms for some values of $n\in [n_1,n_2]$, and we still need to correct for these. Fortunately, there are only finitely many such terms, and it is easy to determine them explicitly. 
To this end, let us specialise to the case whereby $a>b\geq 1$. Then, by Lemma~\ref{lemma:R_0(a,b)}, $R_0(a,b)$ is dominated by $x^ay^b(y^{3a-3}+x^{3b})$, and every nonzero monomial term in $R_0(a,b)$ is of the form $\lambda x^ay^b\cdot x^{3
\ell}y^{3u}$. 

\begin{lemma}\label{lemma:BadTerms}
    Suppose $a>b\geq 1, \ell,u\geq1$ are positive integers. Then, the following are equivalent.
    \begin{enumerate}
        \item We have
        \[
        \val_n(x^{3\ell}y^{3u}) \geq \val_n(y^{3a-3}+ x^{3b}) \textrm{ for every } n\in\left[n_1,n_2\right]
        \]
        and
        \[
         \val_n(x^{3\ell}y^{3u}) < \val_n(y^{3a} + x^{3b}) \textrm{ for some } n\in\left[n_1,n_2\right].
         \]
         \item The pair $(\ell,u)$ belongs to $\Lambda_{a,b}$, where
         \[
         \Lambda_{a,b}=\left\{(m,s)\in\ZZ^2 \ : \ m>0,\ s>0, \   \frac{a}{b}m <a-s\leq \frac{a-1}{b}m+1\right\}.
         \]
    \end{enumerate}
\end{lemma}

\begin{proof}
    Writing $n\in[n_1,n_2]$ as $n=(1-r)n_1+rn_2$ for $r\in[0,1]$, we have 
    \[
        \val_n(x^{3\ell}y^{3u}) \geq \val_n(y^{3a-3}+ x^{3b}) \textrm{ for every } n\in\left[n_1,n_2\right]
    \]
        and
    \[
         \val_n(x^{3\ell}y^{3u}) < \val_n(y^{3a} + x^{3b}) \textrm{ for some } n\in\left[n_1,n_2\right]
    \]
    if and only if
    \[
    3ur+3\ell(1-r)\ge\min\{(3a-3)r,3b(1-r)\} \textrm{ for all } r\in[0,1],
    \]
    and 
    \[
    3ur+3\ell(1-r)<\min\{3ar,3b(1-r)\} \textrm { for some } r\in[0,1].
    \]
    Noting that $3ur+3\ell(1-r)$ depends affine linearly on $r$, and that the maximal values of $\min\{(3a-3)r,3b(1-r)\}$ and $\min\{3ar,3b(1-r)\}$ are respectively $3b(a-1)/(a+b-1)$ and $3ab/(a+b)$ and they are respectively achieved for $r=b/(a+b-1)$ and $r=b/(a+b)$, we see that this happens precisely when 
    \[
    3u\frac{b}{a+b}+3\ell\frac{a}{a+b}<\frac{3ab}{a+b} \quad \textrm{ and } \quad 3u\frac{b}{a+b-1}+3\ell\frac{a-1}{a+b-1}\geq \frac{3b(a-1)}{a+b-1},
    \]
    that is,
    \[
    a-\frac{a-1}{b}\ell - 1\leq u < a-\frac{a}{b}\ell.
    \]
    
\end{proof}
One way to restate the above Lemma is as follows: if a power series $F(x^3,y^3)$ is dominated by $y^{3a-3}+x^{3b}$ (and only involves terms of the form $x^{3p}y^{3q}$), and if the coefficient of $y^{3a-3}$ in $F(x^3,y^3)$ is zero, then it is in fact dominated by 
\[
y^{3a}+x^{3b}+\sum_{(m,s)\in\Lambda_{a,b}} x^{3m}y^{3s}.
\]
The set $\Lambda_{a,b}$ therefore corresponds to the set of potentially `bad terms' in $R_0(a,b)$ which will dominate $x^ay^b(y^{3a}+x^{3s})$ if they occur with nonzero coefficient. The next result essentially tells us that we can subtract off these terms in finitely many steps.
\begin{lemma}
    Suppose $a>b\geq1$. Then, the set $\Lambda_{a,b}$ is a finite set of cardinality
    \[
    k=\frac{b+\gcd(a-1,b)-\gcd(a,b)-1}{2}.
    \]
    In particular, $\Lambda_{a,b}=\varnothing$ if and only if $a$ is a multiple of $b$. Moreover, the maps $\Lambda_{a,b}\to \ZZ, (m,s)\mapsto m$ and $(m,s)\mapsto s-m$ are injective.
\end{lemma}
\begin{proof}
    Observe that $\Lambda_{a,b}$ is a subset of $D_{a,b}\cap \ZZ^2$ where $D_{a,b}$ is the convex set given by
    \[
    D_{a,b}=\left\{(m,s)\in\RR^2 \ : \ m\geq 0, \ s\geq 0, \ \frac{a}{b}m<a-s\leq \frac{a-1}{b}m+1\right\}.
    \]
    But now observe that $D_{a,b}$ is the convex hull of $(0,a-1)$, $(0,a)$ and $(b,0)$. This shows that $\Lambda_{a,b}$ is a finite set. By Pick's Theorem, the number of interior integral points in $D_{a,b}$ is given by
    \[
     \frac{b+1-\gcd(a,b)-\gcd(a-1,b)}{2},
    \]
    and since $\Lambda_{a,b}$ is the union of these interior integral points and the integral points on the interior of the edge joining $(0,a-1)$ and $(b,0)$, we get that $\Lambda_{a,b}$ has cardinality
    \[
    k=\frac{b+1-\gcd(a,b)-\gcd(a-1,b)}{2}+\gcd(a-1,b)-1 = \frac{b+\gcd(a-1,b)-\gcd(a,b)-1}{2}.
    \]
    If $a$ is a multiple of $b$, then $\gcd(a,b)=b, \gcd(a-1,b)=1$, and so $k=0$. Conversely, if $k=0$ then $b\geq \gcd(a,b) = b+\gcd(a-1,b) -1>b-1$, whence $\gcd(a,b)= b$ and so $a$ is a multiple of $b$.
    
    For the last claim, first assume $(m,s_1),(m,s_2)\in \Lambda_{a,b}$. Then we have 
    \[
    |s_1-s_2| < \left\vert\frac{a}{b}m-\frac{a-1}{b}m-1\right\vert=\left\vert1-\frac{m}{b}\right\vert < 1
    \]
    since $0<m<b$. This forces $s_1=s_2$. Next assume $(m,s),(m+e,s+e)\in\Lambda_{a,b}$ for some integer $e \geq 0$. If $e>0$ then we have 
    \[
    a-(s+e)>\frac{a}{b}(m+e)=\frac{a}{b}m+\frac{a}{b}e > \frac{a-1}{b}m+1 \geq a - s
    \]
    which is absurd. Hence, $e=0$.
\end{proof}

If $\Lambda_{a,b} = \varnothing$ (that is, $a$ is a multiple of $b$), then set $S(a,b) = R_0(a,b)$. Otherwise, write 
\[\Lambda_{a,b}=\{(m_1,s_1),(m_2,s_2),\dots,(m_k,s_k)\}
\] 
with $m_1<m_2<\dots<m_k$. Then, we have 
\[
a-s_i\leq\frac{a-1}{b}m_i+1<\frac{a}{b}m_i+\frac{a}{b}=\frac{a}{b}(m_i+1)\leq \frac{a}{b}m_{i+1}<a-s_{i+1}
\]
which shows that $s_1>s_2>\dots > s_k$. Set $d_i=s_i-m_i$. Then we have $0<d_i<a$ and $d_1>d_2>\dots > d_k$.  For $i=1,\dots, k$, define
\begin{equation}\label{eqn:Riab}
R_i(a,b)=
\begin{cases}
    R_{i-1}(a,b) - \lambda_it^{2a+5m_i-2s_i}x_0^{a-d_i}S(2d_i-a,a+b-d_i)& \textrm{ if } 3d_i > 2a,\\
    \\
    R_{i-1}(a,b) - \lambda_it^{a+3m_i}x_0^{2a-3d_i}x_2^{3d_i+b-a} & \textrm{ if } a-b\leq 3d_i \leq 2a,\\
    \\
    R_{i-1}(a,b)- \lambda_it^{b+3s_i}x_0^{2b+3d_i}x_1^{a-b-3d_i}& \textrm{ if } 3d_i < a-b,
    
\end{cases}
\end{equation}
where $\lambda_i$ is the coefficient of $x^ay^b\cdot x^{3m_i}y^{3s_i}$ in the Taylor expansion of $R_{i-1}(a,b)$ in $x$, $y$. Note that when $3d_i>2a$, $R_i(a,b)$ is well defined by the inductive hypothesis, since $(2d_i-a)+(a+b-d_i)=b+d_i<a+b$.
\begin{lemma}
    For each $i=0,1,\dots,k$, the Taylor expansion of $R_i(a,b)$ in $x,y$ is of the form
    \[
    R_i(a,b)=x^ay^b\left(\left(\frac{y^3}{1+y}\right)^{a}+\left(\frac{x^3}{1+x^3}\right)^{b}+x^3y^3G_{i,a,b}(x^3,y^3)\right)
    \]
    where every nonzero monomial term in $x^ay^b\cdot x^3y^3G_{i,a,b}(x^3,y^3)$ is dominated by 
    \[
    f_i=x^ay^b(y^{3a}+x^{3b}+\sum_{j>i}x^{3m_j}y^{3s_j}).
    \]
    In particular, every nonzero monomial term in the Taylor expansion of $R_k(a,b)$ is dominated by $x^ay^b(y^{3a}+x^{3b})$ and the coefficients of the $x^ay^b\cdot y^{3a}$ and $x^ay^b\cdot x^{3b}$ terms are both equal to $1$.

\end{lemma}
\begin{proof}
The statement is true for $i=0$ if we set $G_{0,a,b} = G_{a,b}$ given by Lemma~\ref{lemma:R_0(a,b)} above. Indeed, by Lemma~\ref{lemma:R_0(a,b)}, $x^3y^3G_{a,b}(x^3,y^3)$ is dominated by $y^{3a-3}+x^{3b}$. Then, the only monomial terms which dominate $y^{3a}+x^{3b}$ while being dominated by $y^{3a-3}+x^{3b}$ are precisely the ones occurring in $f_0$ by Lemma~\ref{lemma:BadTerms}. We can assume therefore $0<i\leq k$ and the conclusion holds for $R_{i-1}(a,b)$.

If $3d_i>2a$, then (by the inductive hypothesis) every nonzero monomial term in the Taylor expansion of $t^{2a+5m_i-2s_i}x_0^{a-d_i}S(2d_i-a,a+b-d_i)$ is a multiple of $x^ay^b\cdot x^{3m_i}y^{3a-3s_i+6m_i}$ and the coefficient of $x^ay^b\cdot x^{3m_i}y^{3s_i}$ is equal to $1$. (Note that $3a-3s_i+6m_i > 0$ as $0<s_i<a$.)

If $a-b\leq 3d_i\leq 2a$, then every nonzero monomial term in the Taylor expansion of $t^{a+3m_i}x_0^{2a-3d_i}x_2^{3d_i+b-a}$ is a multiple of $x^ay^b\cdot x^{3m_i}y^{3s_i}$ and the coefficient of $x^ay^b\cdot x^{3m_i}y^{3s_i}$ is equal to $1$.

If $3d_i<a-b$, then every nonzero monomial term in the Taylor expansion of $t^{b+3s_i}x_0^{2b+3d_i}x_1^{a-b-3d_i}$ is once again a multiple of $x^ay^b\cdot x^{3m_i}y^{3s_i}$ and the coefficient of $x^ay^b\cdot x^{3m_i}y^{3s_i}$ is again equal to $1$. 

This shows that in the Taylor expansion of $R_i(a,b)$ the `pure terms'
\[
x^ay^b \left(\left(\frac{y^3}{1+y^3}\right)^a + \left(\frac{x^3}{1+x^3}\right)^b\right)
\]
are the same as for $R_{i-1}(a,b)$, and hence the Taylor expansion of $R_i(a,b)$ is of the form 
\begin{equation}
R_i(a,b)=x^ay^b\left(\left(\frac{y^3}{1+y}\right)^{a}+\left(\frac{x^3}{1+x^3}\right)^{b}+x^3y^3G_{i,a,b}(x^3,y^3)\right)
\end{equation}
for some power series $G_{i,a,b}(x^3,y^3)$. Moreover, it also shows that the coefficient of $x^ay^b\cdot x^{3m_i}y^{3s_i}$ in the Taylor expansion of $R_i(a,b)$ is zero. 

Now, by assumption $R_{i-1}(a,b)$ is dominated by
\[f_{i-1}=x^ay^b(y^{3a}+x^{3b} + \sum_{j\geq i}x^{3m_j}y^{3s_j})
\]
so the conclusion will follow once we establish that, in each of the three cases, the term being added to $R_{i-1}(a,b)$ to obtain $R_i(a,b)$ is also dominated by $f_{i-1}$, for this will imply (since $R_i(a,b)$ has no $x^ay^b(x^{3m_i}y^{3s_i})$ term by construction), that $R_i(a,b)$ is dominated by 
\[f_{i-1}-x^ay^b(x^{3m_i}y^{3s_i})= x^ay^b(y^{3a}+x^{3b} + \sum_{j> i}x^{3m_j}y^{3s_j}) = f_i.\]

If $3d_i>2a$, then by the inductive hypothesis $t^{2a+5m_i-2s_i}x_0^{a-d_i}S(2d_i-a,a+b-d_i)$ is dominated by 
\[
x^{a+3m_i}y^{b+3a-3d_i+3m_i}(y^{3(2d_i-a)}+x^{3(a+b-d_i)}) = x^{a}y^{b}(x^{3m_i}y^{3s_i}+x^{3b}\cdot x^{3(a-d_i+m_i)}\cdot y^{3a-3d_i+3m_i})
\]
which is dominated by $f_{i-1}$ since every one of its monomial terms is dominated by some monomial term of $f_{i-1}$. 

If $a-b\leq 3d_i \leq 2a$, then $t^{a+3m_i}x_0^{3a-3d_i}x_2^{3d_i+b-a}$ is dominated by $x^ay^b\cdot x^{3m_i}y^{3s_i}$ and therefore also by $f_{i-1}$. The case whereby $3d_i<a-b$ is identical.
\end{proof}

Finally, we can define $S(a,b)$ when $a>b\geq1$ and $a$ is not a multiple of $b$ by setting 
\[
S(a,b) = R_k(a,b).
\]

\begin{prop}\label{prop:SabMatchesCost}
     For $n \in \partial\Delta^\vee$, we have 
    \[
    \val_n(S(a,b)) = \frac{a+b}{3}c\left(\frac{a}{a+b}m_0+\frac{b}{a+b}m_1,n \right).
    \]
\end{prop}
\begin{proof}
    By construction the Taylor expansion of $S(a,b)$ in $x=x_1/x_0, y = x_2/x_0$ is dominated by $x^ay^b(y^{3a}+x^{3b})$ and these terms occur in $S(a,b)$ with coefficient equal to 1. Thus, the equality holds (by construction) for $n\in [n_1,n_2]$. Therefore, it only remains to establish the equality when $n\in [n_0,n_1]$ or $n\in[n_2,n_0]$. 
    
    Let $n\in[n_0,n_1]$. By the inductive hypothesis, the Taylor expansion of $S(p,q)$ in $\xi=x_0/x_2, \eta = x_1/x_2$ is dominated by $\eta^p$. From this, it follows that the Taylor expansion of $R_0(a,b)$ is dominated by $\eta^a+\eta^{a}\xi^3$ if $a>b \geq 1$ and therefore by $\eta^a$. Now suppose $R_{i-1}(a,b)$ is dominated by $\eta^a$. If $3d_i>2a$, then $R_i(a,b)$ is dominated by $\eta^a + \eta^{2a+5m_i-2s_i+2d_i-a}\xi^{a-d_i} = \eta^a(1+\eta^{3m_i}\xi^{a-d_i})$. If $a-b
    \leq 3d_i\leq 2a$, then $R_i(a,b)$ is dominated by $\eta^a+\eta^{a+3m_i}\xi^{2a-3d_i}$ and if $3d_i<a-b$ then $R_i(a,b)$ is dominated by $\eta^a+\eta^{b+3s_i+a-b-3d_i}\xi^{2b+3d_i}=\eta^{a}(1+\eta^{3m_i}\xi^{2b+3d_i})$. In all three cases, $R_i(a,b)$ is dominated by $\eta^a$. This shows that $S(a,b)$ is dominated by $\eta^a$. 
    
    Now, it is easy to see by the inductive hypothesis that $S(a,b) = x_1^ax_2^b+tS^\prime(a,b)$ for some section $S^\prime (a,b) \in H^0(X,\mathcal O(a+b))$, so the coefficient of $\eta^a$ in the Taylor expansion of $S(a,b)$ in $\xi,\eta$ is equal to $1$. This shows that for $n\in[n_0,n_1]$, we have 
    \[
    \val_n(S(a,b))= \val_n(\eta^a).
    \]
    This establishes the equality for $n\in[n_0,n_1]$. The case $n\in[n_2,n_0]$ is exactly the same, and we omit it to avoid repetition.
\end{proof}

This defines $S(a,b)$ whenever $a>b\geq 1$. We define $S(a,b)$ when $b>a\geq 1$ by interchanging the roles of $a$ and $b$ and $x_1$ and $x_2$ in the above construction. More precisely, recall the notation $S(a,b) = S(a,b,x_0,x_1,x_2)$ indicating how $S(a,b)$ depends on  $x_0,x_1,x_2$ explicitly. Then, if $b>a\geq 1$, we define
\[
S(a,b,x_0,x_1,x_2) = S(b,a,x_0,x_2,x_1).
\]
This process defines $\theta^{(a+b)/3}_m$ whenever $m\in[m_0,m_1]\cap(m_0+\frac{a+b}{3}\ZZ^2)$. By symmetry, we may therefore define
\[
\theta^{(a+b)/3}_m = \begin{cases}
    S(a,b,x_0,x_1,x_2) & \textrm{ if }\quad m=\frac{a}{a+b}m_0+\frac{b}{a+b}m_1,\\
    \\
    S(a,b,x_1,x_2,x_0) & \textrm{ if }\quad m=\frac{a}{a+b}m_1+\frac{b}{a+b}m_2,\\
    \\
    S(a,b,x_2,x_0,x_1) & \textrm{ if }\quad m=\frac{a}{a+b}m_2+\frac{b}{a+b}m_0.
\end{cases} 
\]
\begin{theorem}
    For $\ell\in\ZZ_{>0}$ and $m\in\partial\Delta\cap\ell^{-1}\ZZ^2 = \partial\Delta_\ell$ denote by $\sigma^\ell_m \in H^0(X_K,L_K^\ell)$ the (restriction of the) monomial section corresponding to $m$. Then, the following statements hold.
\begin{enumerate}[(i)]
    \item For each $k\geq 1$, the collection $\{\theta^k_m\}_{m\in \partial\Delta_k}$ is a valuatively independent $\CC( \! (t) \! )$-basis of $H^0(X,L^k)$ and for $n\in \partial\Delta^\vee$ and $m\in\partial\Delta_k$ we have 
    \begin{equation}\label{eqn:tropthetaequalscost}
    k^{-1}\val_n(\theta^k_m) = c(m,n).
    \end{equation}
    \item For each $k\geq 1$, there exists a linear map $S_k \in GL(H^0(X,L^k), \CC( \! (t) \! ))$ such that 
        \[
        \theta^k_m = \sigma^k_m + tS_k(\sigma^k_m).
        \]
         The coefficients of $S_k$ in the monomial basis $\{\sigma^k_m\}$ are formal power series in $t$ with integer coefficients.
        \item For each $k, \ell \geq 1$, and each $m \in \partial \Delta_k \subseteq \partial \Delta_{k\ell}$, we have 
        \[
        \frac{1}{k\ell}\val_n(\theta^{k\ell}_m) = \frac{1}{k}\val_n (\theta^k_m)
        \]
        for each $n\in \op{Sk}(X,\pi)$. 
        \item Let \[B \subseteq C^0(\op{Sk}(X,\pi),\RR)\] denote the closure of the set of normalised tropical theta functions in the normed space $C^0(\op{Sk}(X,\pi),\RR)$. There is a homeomorphism 
        \[ \Psi : \partial\Delta \longrightarrow B\]
        such that $\Psi(m)=\hat\theta_m$ whenever $m \in \partial\Delta\cap\QQ^2$ is a rational point.
        \item For every $m\in\partial\Delta$, the function $\Psi(m) \in B\subset C^0(\op{Sk}(X,\pi),\RR)$ satisfies the real-Monge-Ampère equation 
        \[\op{MA}(-\Psi(m))~=~9\delta_{F(m)},\] 
        where $F: \partial \Delta \to \partial \Delta^\vee$ is given by
        \begin{equation}\label{eqn:Fdef}
        F(m) = \frac{1}{3}\sum_{v\in\op{Vert(\Delta^\vee))}} (1 - \langle v, m\rangle) v \in \partial\Delta^\vee.
        \end{equation}
\end{enumerate}
\end{theorem}
\begin{proof}
    By Proposition \ref{prop:SabMatchesCost}, we have $k^{-1}\val_n(\theta^k_m) = c(m,n)$ for each $n\in \partial\Delta^\vee$. Thus, (i) follows by Proposition \ref{prop:HessianCostisValuative}.
    Moreover, by construction, $\theta^k_m = \sigma^k_m + ts^k_m$ for some section $s^k_m$. Since the monomials $\{ \sigma^k_{m^\prime}\}_{m^\prime \in \partial\Delta_k}$ comprise a $\CC( \! (t) \! )$-basis of $H^0(X,L^k)$ and since the expressions \eqref{eqn:R0ab} and \eqref{eqn:Riab} in the inductive construction of $\theta^k_m$ only involve linear combinations of monomials with integer coefficients, (ii) follows. Item (iii) follows by \eqref{eqn:tropthetaequalscost} in (i), noting that $c(m,\cdot)$ only depends on $m\in \partial\Delta_k \cap \partial\Delta_{k\ell}$ and not on $k$ or $\ell$. For (iv), observe that $\partial\Delta \to C^0(\op{Sk}(X,\pi),\RR) , m\mapsto c(m,\cdot)$ is a continuous map from a compact space to a Hausdorff space. It is easy to verify that this map is injective. (For example, $c(m,\cdot)$ achieves its maximum value only at the point $F(m)$ given by \eqref{eqn:Fdef} and $F$ is injective.)  It is therefore a homeomorphsim onto its image. Since its image is compact, it is a closed subset of $C^0(\op{Sk}(X,\pi),\RR)$, and it takes the dense subset $\partial\Delta\cap\QQ^2$ to the set of normalised tropical theta functions. This shows that its image agrees with the closure of the set of normalised tropical theta functions, proving (iv). Finally, (v) follows from \eqref{eqn:tropthetaequalscost} and Corollary \ref{cor:MA}.
    
\end{proof}

As mentioned in the introduction, although the basis $\theta^\ell_m$ itself is not uniquely determined by its being valuatively independent, the functions $n\mapsto \val_n(\theta^\ell_m)$ are all uniquely determined. Indeed, suppose $\tilde\theta^\ell_m$ were another basis such that $\tilde \theta^\ell_m = \sigma^\ell_m + t\tilde \theta_m^\prime$ for some section $\tilde \theta^\prime_m\in H^0(X,L^\ell)$. Then, $\theta^\ell_m=(I+tB)\sigma^\ell_m$ and $\tilde \theta_m= (I+t\tilde B)\sigma_m$, whence we have that 
\[
\tilde \theta^\ell_m = (I+tB^\prime)\theta^\ell_m. 
\]
Writing $B^\prime = (b_{m m^\prime})$, we then have 
\[
\tilde \theta^\ell_m = (1+tb_{mm}) \theta^\ell_m + \sum_{m^\prime\neq m} b_{mm^\prime}\theta^\ell_{m^\prime}.
\]
Thus, we have 
\[
\val_n(\tilde \theta^\ell_m) = \min\{\val_t(1+tb_{mm})+\val_n(\theta^\ell_m), \min_{m^\prime\neq m} \{\val_t(b_{mm^\prime}) +\val_n(\theta^\ell_{m^\prime})\}\} \leq \val_n(\theta^\ell_m)
\]
since $\val_t(1+tb_{mm})=0$. Similarly, $\val_n(\theta^\ell_m)\leq \val_n(\tilde \theta^\ell_m)$, so equality follows.

\end{document}